\documentclass[psamsfont,a4paper,12pt]{amsart}

\usepackage[english]{babel}   %loads english
\numberwithin{equation}{section}

\usepackage{setspace}     %controls line-spacing
\usepackage{enumerate} 	%allows lists
\usepackage{afterpage} 	%keeps tables and figures where they belong
\usepackage[usenames,dvipsnames]{color}
\usepackage{graphicx}	%allows inclusion of pictures

\usepackage{amsmath} 	%just apparently the most important thing ever
\usepackage{amssymb}	%symbols in math
\usepackage{mathrsfs} 	%really fancy script font
\usepackage{amsfonts} 	%standard fonts
\usepackage{latexsym} 	%more symbols
\usepackage[latin1]{inputenc} %sets up proper special characters

\usepackage{amsthm} 	%sets up theorem styling
\usepackage{stackrel}       %allows stacking in mathmode  
\usepackage{amscd} 	%commutative diagrams                                 

\usepackage{tabularx}	%automatically widening tables
\usepackage{longtable}	%for tables longer than a page

\usepackage{verbatim} %comment environment?
\usepackage{blindtext} %something about dummy text

\usepackage{amssymb,latexsym,amsmath}    % Standard packages
\usepackage{pstricks,multido,pst-plot}		 % Packages for drawing figures
\usepackage{multicol,fancyheadings}			   % Misc.
\usepackage{float}                       % Used to define Algorithm enviro
\usepackage[vcentermath,enableskew]{youngtab}

\usepackage{subfig}

\allowdisplaybreaks

\def\b{\beta} 
\def\d{\delta}

\newcommand{\C}{\mathbb{C}}

\newcommand{\Z}{\mathbb{Z}}

\def\Tr{\mathop{\rm Tr}}

\newcommand{\genstirlingI}[3]{%
  \genfrac{[}{]}{0pt}{#1}{#2}{#3}%
}
\newcommand{\stirlingI}[2]{\genstirlingI{}{#1}{#2}}

\newtheorem{theorem}{Theorem}[section]  %%with section numbering

\newtheorem{lemma}[theorem]{Lemma}

\theoremstyle{definition}

\newtheorem*{defn}{Definition}

\theoremstyle{remark}
 
\newtheorem*{rmk}{Remark}

\newtheorem{ind}[]{{\rm\it Indice}}

\usepackage{multicol}

\usepackage{tikz}
\usetikzlibrary{arrows}

\usepackage{paralist}
\usepackage{cite}

\title{Higher Width Moonshine}

\author[Dawsey]{Madeline Locus Dawsey*}
\address{Department of Mathematics,
University of Texas at Tyler, Tyler, TX 75799}
\email{mdawsey@uttyler.edu}

\author[Ono]{Ken Ono}
\address{Department of Mathematics, University of Virginia, Charlottesville, VA 22904}
\email{ken.ono691@virginia.edu}

\begin{document}

\thanks{*This author was previously known as Madeline Locus.}

\begin{abstract}  {\it Weak moonshine} for a finite group $G$ is the phenomenon where an infinite dimensional
 graded $G$-module $$V_G=\bigoplus_{n\gg-\infty}V_G(n)$$ has the property that its trace functions, known as McKay-Thompson series, are modular functions.  Recent work by DeHority, Gonzalez, Vafa, and Van Peski established that weak moonshine holds
 for every finite group.
 Since weak moonshine only relies on character tables, which are not isomorphism class invariants, non-isomorphic groups can have the same McKay-Thompson series.  We address this problem by extending weak moonshine to arbitrary width $s\in\mathbb{Z}^+$.  For each $1\leq r\leq s$ and each irreducible character $\chi_i$, we employ Frobenius' $r$-character extension $\chi_i^{(r)} \colon G^{(r)}\rightarrow\mathbb{C}$ to define {\it width $r$ McKay-Thompson series} for $V_G^{(r)}:=V_G\times\cdots\times V_G$ ($r$ copies) for each $r$-tuple in $G^{(r)}:=G\times\cdots\times G$ ($r$ copies).  These series are modular functions which then reflect differences between $r$-character values.   Furthermore, we establish orthogonality relations for the Frobenius $r$-characters, which dictate the compatibility of the extension of weak moonshine for $V_G$ to width $s$ weak moonshine.
\end{abstract}

\maketitle

\section{Introduction and Statement of Results}

The {\it Monstrous Moonshine Conjecture} of Conway and Norton \cite{CN}  offered a surprising relationship between the largest sporadic simple group, the monster $\mathbb{M},$ and modular functions.  The conjecture extended the observation of McKay and Thompson that the first few coefficients of 
\begin{equation}\label{J-function}
J(\tau):=j(\tau)-744=q^{-1}+196884q+21493760q^2+O\left(q^3\right),
\end{equation}
the Hauptmodul for $\mathrm{SL}_2(\mathbb{Z})$ with $q:=e^{2\pi i\tau}$ and $\tau\in\mathbb{H}$, are simple sums of the dimensions of the 194 irreducible representations of $\mathbb{M}$.  For example, if we let $\chi_1,\chi_2,\chi_3$ denote the first three (ordered by dimension) irreducible characters of $\mathbb{M}$, then we have
\begin{align*}
1&=\chi_1(1),\\
196884&=1+196883=\chi_1(1)+\chi_2(1),\\
21493760&=1+196883+21296876=\chi_1(1)+\chi_2(1)+\chi_3(1).
\end{align*}
Thompson \cite{Thompson1,Thompson2} made the conjecture that there is a graded, infinite-dimensional $\mathbb{M}$-module $$V^\natural=\bigoplus_{n\geq-1} V^\natural(n)$$ whose graded dimensions $\dim V^\natural(n)$ are the Fourier coefficients of $J(\tau)$.  Conway and Norton \cite{CN} then formulated the overarching {\it Monstrous Moonshine Conjecture}, the assertion that, for each $g\in\mathbb{M}$, there is a genus zero subgroup $\Gamma_g\subseteq\mathrm{SL}_2(\mathbb{R})$ such that the graded trace function 
$$
T_g(\tau):=\sum_{n\geq-1}\mathrm{Tr}\left(g\vert V^\natural(n)\right)q^n,
$$
called the \emph{McKay-Thompson series}, is the Hauptmodul for $\Gamma_g$.  
Borcherds \cite{Borcherds} famously proved this conjecture in 1992.

In the aftermath of Borcherds' work, further examples of ``moonshine'' have been obtained. 
Monstrous moonshine was extended, giving rise to Norton's Generalized Moonshine conjecture 
and Ryba's Modular Moonshine conjecture (see \cite{Carnahan}).
The McKay-Thompson series have even been mock modular forms, illustrating that there are many more types of moonshine than mathematicians and physicists initially thought.  Duncan, Griffin, and the second author gave a survey  \cite{DGO1} of the advancements in the theory of moonshine and its applications to physics as of 2015, including their proof \cite{DGO2} of the \emph{Umbral Moonshine Conjecture} of Cheng, Duncan and Harvey
\cite{CDH}. Along these lines, there has been a flurry of recent work (for example, see \cite{DMO,GM}).

In 2017, DeHority, Gonzalez, Vafa, and Van Peski \cite{REU} examined the question of the extent to which dimensions of irreducible representations of finite groups and Fourier coefficients of modular functions are related.  They proved (see Theorem 1.1 of \cite{REU}) that the seemingly rare phenomenon of moonshine  holds for every single finite group if we relax certain requirements.  Namely, for every finite group $G$ there is an infinite-dimensional graded $G$-module
\begin{equation}\label{module}
V_G=\bigoplus_{n\in\{-d\}\cup\mathbb{Z}^+}V_G(n),
\end{equation}
for sufficiently large $d>0$, such that the McKay-Thompson series for each $g\in G$ is a weakly holomorphic\footnote{A weakly holomorphic modular function is allowed to have poles at cusps.} modular function.  We refer to this generalization as \emph{weak moonshine}.

\begin{rmk}
Monstrous moonshine, and other
strong examples of moonshine, are equipped with a rich algebra structure, typically as  vertex operator algebras.
Recent work by Evans and Gannon \cite{EvansGannon} offers such moonshine for any finite solvable group $G$ and cohomology class in $H^4(G,\Z)$.
\end{rmk}

Unfortunately, two non-isomorphic groups can have the same moonshine.  This arises from the fact that weak moonshine depends only on character tables, which do not uniquely determine a group (for example, consider the dihedral group $D_4$ and the quaternion group $Q_8$). This problem is acute for {\it Brauer pairs}, pairs of non-isomorphic finite groups which admit an isomorphism of character tables that preserves power maps on conjugacy classes. A classic theorem of Dade \cite{Dade} proves that there are
infinitely many Brauer pairs.\footnote{Dade's theorem offers infinitely many Brauer pairs among $p$-groups.}
Therefore, it is natural to ask for extensions of weak moonshine that distinguish such groups.

To answer this question, we make use of the generalized, ``higher width" group characters defined by Frobenius in \cite{Frob}.
Let $G$ be a finite group, and let $\rho_1,\dots,\rho_t$ be the irreducible representations of $G$; i.e., each $\rho_i$ is a group homomorphism $\rho_i \colon G\rightarrow\mathrm{GL}(V_i)$ for some $\mathbb{C}$-vector space $V_i$.  Let $\chi_1,\dots,\chi_t$ be the irreducible characters of $G$, which are the class functions $\chi_i \colon G\rightarrow\mathbb{C}$ defined by $\chi_i(g):=\mathrm{Tr}\left(\rho_i(g)\right)$ for all $g\in G$.  

We now turn to the Frobenius $r$-characters.
For $r\in\mathbb{Z}^+,$ we let
$G^{(r)}:=G\times\cdots\times G\hspace{.2cm}(r\text{ copies}).$
 If $\chi$ is an irreducible character, then its $r$-character generalizations are defined by letting $\chi^{(1)}(g):=\chi(g)$, $\chi^{(2)}\left(g_1,g_2\right):=\chi\left(g_1\right)\chi\left(g_2\right)-\chi\left(g_1g_2\right),$ and for $r\geq3$ by the recursive relation
\begin{align}\label{defn}
\chi^{(r)}\big(g_1,&\dots,g_r\big):=\chi\left(g_1\right)\chi^{(r-1)}\left(g_2,\dots,g_r\right)\\
&-\chi^{(r-1)}\left(g_1g_2,\dots,g_r\right)-\chi^{(r-1)}\left(g_2,g_1g_3,\dots,g_r\right)-\cdots-\chi^{(r-1)}\left(g_2,\dots,g_1g_r\right).\nonumber
\end{align}

For many years, the problem of determining the role of the Frobenius $r$-characters in group theory remained open.  Namely, to what extent do Frobenius $r$-characters uniquely determine a group up to isomorphism?  Hoehnke and Johnson \cite{HJ1,HJ2} gave the very satisfying answer that a group is uniquely determined by its 1, 2, and 3-characters.  Therefore, the goal here is to construct an extension of weak moonshine that also makes use of the 2 and 3-characters. It turns out to be quite simple.

We proceed with this goal in mind.  Suppose that $G$ satisfies weak moonshine
with $V_G$ as in (\ref{module}).
For $1\leq i \leq t,$ we let  $m_i(n)$ denote the multiplicity of the representation space for $\rho_i$ in $V_G(n)$.
For $g\in G$, weak moonshine asserts that
 the McKay-Thompson series
 $$
 T_g(\tau):=\sum_{n\gg -\infty} \Tr(g| V_G(n))q^n=\sum_{n\gg -\infty} \sum_{1\leq i \leq t} m_i(n)\chi_i(g)q^n
 $$
 are modular functions.
 
To extend this moonshine, we assemble {\it width $r$ McKay-Thompson series} using $r$-characters.
 If $1\leq r\leq s$ and $\underline{g}:=\left(g_1,\dots,g_r\right)\in G^{(r)}$, then we define the \emph{$r$-Frobenius of $\underline{g}$ on $V_G^{(r)}(n):=V_G(n)\times\cdots\times V_G(n)$} ($r$ copies) by
\begin{equation}\label{r-Frob}
\mathrm{Frob}_r\left(\underline{g};n\right):=\sum_{1\leq i\leq t}m_i(n)\chi_i^{(r)}\left(\underline{g}\right).
\end{equation}
 For each $1\leq r\leq s$ and each $\underline{g}\in G^{(r)}$, we define the width $r$ \emph{McKay-Thompson series}
\begin{equation}\label{T}
T\left(r,\underline{g};\tau\right):=\sum_{n\gg-\infty}\mathrm{Frob}_r\left(\underline{g};n\right)q^n.
\end{equation}

\begin{defn}
We say that $G$ has \textbf{width $s\geq1$ weak moonshine} if for each $1\leq r\leq s$ and each $\underline{g}\in G^{(r)}$ we have that
 $T\left(r,\underline{g};\tau\right)$ is a weakly holomorphic modular function.
\end{defn}

\begin{rmk}
If $r=1$, then
we have  $\mathrm{Frob}_1(g;n)=\mathrm{Tr}\left(g\vert V_G(n)\right).$  In particular, if $g=e$ is the identity, then the graded dimensions $\mathrm{dim}V_G(n)$ are the coefficients of $T(1,e;q)=T_e(\tau)$.
\end{rmk}

Weak moonshine is {\it complete} if for each $1\leq i\leq t$ there is a nonzero $m_i(n)$.   Thanks to standard facts about modular functions, complete weak moonshine has the property, for each $i$, that $m_i(n)>0$ for infinitely many $n$. The deepest examples of moonshine are asymptotically regular, a feature of monstrous moonshine
which was confirmed in  2015 \cite{DGO1}. A moonshine module $V_G$ is {\it asymptotically regular} if for each $1\leq i\leq t$ we have that
\begin{equation}
\lim_{n\rightarrow \infty} \frac{m_i(n)}{\sum_{j=1}^{t} m_j(n)} =\frac{\dim \chi_i}{\sum_{j=1}^{t} \dim \chi_j}.
\end{equation}

We obtain the following theorem regarding the existence of width $s\geq1$ weak moonshine.

\begin{theorem}\label{moonshine}
If $G$ is a finite group and $s\in \Z^{+}$, then weak moonshine for $G$ extends to width $s$ weak moonshine.  Moreover, 
 $G$
admits asymptotically regular width $s$ weak moonshine for every $s\in \Z^{+}$.
\end{theorem}

\begin{rmk}
We note that $\chi_i^{(r)}$ vanishes if $r>\dim(\chi_i)$. Therefore, these moonshine modules
are trivial for sufficiently large $s$.
\end{rmk}

\begin{rmk}
We can replace weakly holomorphic modular functions with  weakly holomorphic modular forms of arbitrary weight $k$, in which case we refer to this moonshine as \emph{width $s$ and weight $k$ weak moonshine}.  Furthermore, we can choose each $T(1,g;\tau)$ to be on the congruence subgroup $\Gamma_0(\mathrm{ord}(g))$, where $\mathrm{ord}(g)$ is the order of $g$ in $G$.  
%This requirement is natural for Brauer pairs for these groups have common group element orders among their conjugacy classes.
\end{rmk}

\begin{rmk}
It would be very interesting to refine the notion of higher width moonshine that allows one to reconstruct character tables from the Fourier expansions of McKay-Thompson series. The results described here
are insufficient in this regard.
\end{rmk}

It is important to understand the algebraic compatibility of the higher width McKay-Thompson series $\big($i.e. $T\left(r,\underline{g};\tau\right)$ with $r\geq2\big)$ under these extensions.  In particular, these series should satisfy relations which reveal the structure of the seed module $V_G$.  In this direction, we turn to the problem of computing the multiplicities $m_i(n)$ of the representation spaces for the irreducible representations $\rho_i$ in the $n$th graded components $V_G(n)$.  The following theorem illustrates the compatibility of weak moonshine for $V_G$ when extended to width $s$.  In short, the multiplicity generating functions are compatible with the McKay-Thompson series $T\left(r,\underline{g};\tau\right)$.

\begin{theorem}\label{multiplicities}
Suppose that width $s$ weak moonshine holds for a finite group $G$ with irreducible characters $\chi_1,\dots,\chi_t$ and McKay-Thompson series $$\left\{T\left(r,\underline{g};\tau\right) : 1\leq r\leq s\text{ and }\underline{g}\in G^{(r)}\right\}.$$
If $1\leq r\leq s$ and  $\mathrm{dim}\chi_i\geq r$, then the multiplicity generating function for $\rho_i$ in $V_G$ satisfies 
$$\mathcal{M}_i(\tau):=\sum_{n\gg-\infty} m_i(n)q^n=\frac{\left(\mathrm{dim}\chi_i\right)^{r-1}}{r!|G|^r\left(\mathrm{dim}\chi_i-1\right)\cdots\left(\mathrm{dim}\chi_i-(r-1)\right)}\sum_{\underline{g}\in G^{(r)}}\overline{\chi_i^{(r)}\left(\underline{g}\right)}T\left(r,\underline{g};\tau\right).$$
\end{theorem}

Theorem \ref{multiplicities} is obtained from new general results on the orthogonality of the Frobenius $r$-characters.  In Section \ref{r-chars} we develop these relations, completing work initiated by Frobenius, Hoehnke, and Johnson \cite{Frob,HJ1,HJ2,J}. These results are of independent interest in character theory.

In Section \ref{moonshineproofs} we prove Theorem \ref{moonshine} and Theorem \ref{multiplicities}.  In the last section we examine a coincidental weak moonshine for $D_4$ and $Q_8$, and we illustrate how its width 2 extension distinguishes these groups.

\section*{Acknowledgements}
The authors thank John Duncan, Michael Mertens, Larry Rolen, Matt Tyler and the referees for their helpful discussions.  The second author thanks the National Science Foundation and the Asa Griggs Candler Fund.

\section{Orthogonality of Frobenius $r$-characters}\label{r-chars}

Throughout, let $G$ be a finite group, and let $\rho_1,\dots,\rho_t$ and $\chi_1,\dots,\chi_t$ be as above.  Classical work of Schur (for example, see \cite{CR}) asserts that if $\chi$ is nontrivial, then
\begin{align}
\label{1}\sum_{g\in G}\chi(g)&=0,
\end{align}
and offers the following orthogonality relations:
\begin{align}
\label{2}\sum_{g\in G}\chi_i(g)\overline{\chi_j(g)}&=|G|\d_{ij},
\end{align}
where $\d_{ij}$ is the usual Kronecker delta function.

It is a natural problem to determine the orthogonality relations of the Frobenius $r$-characters for $r>1$.  Frobenius, Hoehnke, and Johnson \cite{Frob,HJ1,HJ2,J} obtained some parts of this theory.  Here we obtain the remaining relations, results which are of independent interest.
Generalizing (\ref{2}), we obtain the full orthogonality relations.

\begin{theorem}\label{general}
If $G$ is a finite group with irreducible characters $\chi_1,\dots,\chi_t$ and $1\leq i,j\leq t$, then for any $r\geq1$ we have that
\begin{align*}
\sum_{\underline{g}\in G^{(r)}}\chi^{(r)}_i\left(\underline{g}\right)\overline{\chi^{(r)}_j\left(\underline{g}\right)}=\frac{r!|G|^r\d_{ij}}{\left(\mathrm{dim}\chi_i\right)^{r-1}}\left(\mathrm{dim}\chi_i-1\right)\cdots\left(\mathrm{dim}\chi_i-(r-1)\right).
\end{align*}
\end{theorem}

\begin{rmk}
If $r=1$, then we consider the product $\left(\mathrm{dim}\chi_i-1\right)\cdots\left(\mathrm{dim}\chi_i-(r-1)\right)$ to be empty, and we set the empty product to be 1.  This gives the usual 1-character relation (\ref{2}).
\end{rmk}

\begin{rmk}
Theorem \ref{general} when $i\neq j$ was obtained earlier by Johnson (see Theorem 2.1 of \cite{J}).
\end{rmk}

\subsection{Some lemmas}\label{lemmas}

We require preliminary lemmas about characters to prove Theorem \ref{general}.

\begin{lemma}\label{aux}
If $\chi_i$ is an irreducible character of $G$, and  $h_1,h_2\in G$, then the following are true:
\begin{enumerate}
\item We have that $$\sum_{g\in G}\chi_i\left(gh_1g^{-1}h_2^{-1}\right)=\frac{\chi_i\left(h_1\right)\overline{\chi_i\left(h_2\right)}|G|}{\mathrm{dim}\chi_i}.$$
\item If $\chi_j$ is an irreducible character of $G$, then we have that $$\sum_{g\in G}\chi_i\left(h_1g\right)\overline{\chi_j\left(gh_2\right)}=\frac{\chi_i\left(h_1h_2^{-1}\right)|G|\d_{ij}}{\mathrm{dim}\chi_i}.$$
\end{enumerate}
\end{lemma}

Lemma \ref{aux} \textit{(1)} was proved by Feit in \cite[equation (5.5)]{F}.  We now recall a classical result of Schur which will aid in the proof of Lemma \ref{aux} \textit{(2)}.  The representations $\rho_1,\rho_2,\dots,\rho_t$ can be viewed as matrix representations $\rho_i \colon G\rightarrow\mathrm{GL}_{m_i}(\mathbb{C}),$ where $m_i=\mathrm{dim}\rho_i$ for each $1\leq i\leq t$.  Namely, for each $1\leq i\leq t$ and each $g\in G$, there is a corresponding matrix
\begin{equation}
\rho_i(g)=:\Big[a_{jk}^{(i)}(g)\Big]_{1\leq j,k\leq m_i}.
\end{equation}
In particular, the image of the character $\chi_i$ for all $g\in G$ is given as the trace
\begin{equation}
\chi_i(g)=\sum_{1\leq j\leq m_i}a_{jj}^{(i)}(g).
\end{equation}

The following classic result of Schur (see Chapter 5 of \cite{CR}) gives the key relationship between two representations which leads to all of the orthogonality relations between two characters.

\begin{lemma}[Schur's Lemma]\label{Schur}
Let $G$ be a finite group, and let $V$ and $W$ be $\C$-vector spaces underlying ordinary irreducible representations of $G$.  If $f \colon V\rightarrow W$ is a $G$-linear map, then $f$ is a scalar multiple of the identity map if $V\cong W$ and $f=0$ if $V\not\cong W$.
\end{lemma}

In preparation for the proof of Lemma \ref{aux} \textit{(2)}, we let $\chi_i$ and $\chi_j$ be irreducible characters of $G$, we let $C$ be an arbitrary $m_i\times m_j$ matrix, and we define the matrix $B_C$ by
\begin{equation}\label{matrixstuff}
B_C:=\sum_{g\in G}\rho_i(g)C\rho_j\left(g^{-1}\right).
\end{equation}
Since $\rho_i$ and $\rho_j$ are homomorphisms, it follows easily  that $\rho_i(h)B_C=B_C\rho_j(h)$ for all $h\in G$.  Therefore, by Schur's Lemma we have that 
\begin{equation}\label{B}
B_C=\begin{cases}0,&\mbox{if }i\neq j,\\b_i(C)\cdot I&\mbox{if }i=j,\end{cases},
\end{equation}
 where $b_i(C)\in\mathbb{C}$, and $I$ is the identity matrix of rank $\mathrm{dim}\chi_i$.  

\begin{proof}[Proof of Lemma \ref{aux} \textit{(2)}]
To prove the claim, we shall make repeated use of (\ref{matrixstuff}) and  (\ref{B}).
Choose an arbitrary $m_i\times m_j$ matrix $C$, and denote its entries by $C:=\left[c_{st}\right]$.  We observe that
\begin{equation*}
\sum_{g\in G}\sum_{1\leq s\leq m_i}\sum_{1\leq t\leq m_j}a_{ws}^{(i)}(g)c_{st}a_{tz}^{(j)}\left(g^{-1}\right)=b_i(C)\d_{ij}\d_{wz}.
\end{equation*}
Since $B_C$ is a diagonal matrix, if $C=C(x,y)=\left[c_{st}\right]$ is defined for given $x\leq m_i,y\leq m_j$ by $c_{st}=\delta_{sx}\delta_{ty}$, then we have that
\begin{equation}\label{basic}
\sum_{g\in G}a_{wx}^{(i)}(g)a_{yz}^{(j)}\left(g^{-1}\right)=b_i(C(x,y))\d_{ij}\d_{wz}.
\end{equation}
Obviously, if $i\neq j$, then this expression vanishes, and so we consider the case where $i=j$, and this becomes
\begin{equation*}
\sum_{g\in G}a_{wx}^{(i)}(g)a_{yz}^{(i)}\left(g^{-1}\right)=b_i\left(C(x,y)\right)\d_{wz}.
\end{equation*}
The constant $b_i\left(C(x,y)\right)$ seems to depend on the choice of $x$ and $y$.  However, notice by replacing $g$ by $h^{-1}$, this gives
\begin{equation*}
\sum_{h\in G}a_{yz}^{(i)}(h)a_{wx}^{(i)}\left(h^{-1}\right)=b_i\left(C(x,y)\right)\d_{wz}=b_i\left(C(w,z)\right)\d_{xy},
\end{equation*}
which holds for all $x,y,w,$ and $z$.  The $\delta_{xy}$ term on the right hand side forces the left hand side to be zero unless $x=y$, in which case $b_i(C(x,y))=b_i(C(x,x))$.  The $\delta_{wz}$ term similarly implies that $b_i(C(w,z))=b_i(C(w,w))$.  If $b_i(C(w,w))=b_i(C(x,x))$ for all $x$ and all $w$, then $b_i(C)$ must be a constant which depends only on $\chi_i$.

Returning to the general case where $i$ might not equal $j$, since $\rho_i$ is a homomorphism, we have
\begin{equation}\label{matrixmult}
a_{sx}^{(i)}\left(h_1g\right)=\sum_{1\leq w\leq m_i}a_{sw}^{(i)}\left(h_1\right)a_{wx}^{(i)}(g).
\end{equation}
We multiply (\ref{basic}) by $a_{sw}^{(i)}\left(h_1\right)$ and sum on $w$ to obtain
\begin{eqnarray*}
\sum_{g\in G}a_{yz}^{(j)}\left(g^{-1}\right)\sum_{1\leq w\leq m_i}a_{sw}^{(i)}\left(h_1\right)a_{wx}^{(i)}(g)=b_i(C)\d_{ij}\d_{xy}\sum_{1\leq w\leq m_i}\d_{wz}a_{sw}^{(i)}\left(h_1\right),
\end{eqnarray*}
which by (\ref{matrixmult}) gives
\begin{equation}\label{basic2}
\sum_{g\in G}a_{sx}^{(i)}\left(h_1g\right)a_{yz}^{(j)}\left(g^{-1}\right)=a_{sz}^{(i)}\left(h_1\right)b_i(C)\d_{ij}\d_{xy}.
\end{equation}
Similarly, we observe that
\begin{equation}\label{matrixmult2}
a_{tz}^{(j)}\left(h_2^{-1}g^{-1}\right)=\sum_{1\leq y\leq m_j}a_{ty}^{(j)}\left(h_2^{-1}\right)a_{yz}^{(j)}\left(g^{-1}\right),
\end{equation}
so we multiply (\ref{basic2}) by $a_{ty}^{(j)}\left(h_2^{-1}\right)$ and sum on $y$ to obtain
\begin{equation}\label{basic3}
\sum_{g\in G}a_{sx}^{(i)}\left(h_1g\right)a_{tz}^{(j)}\left(h_2^{-1}g^{-1}\right)=a_{sz}^{(i)}\left(h_1\right)a_{tx}^{(j)}\left(h_2^{-1}\right)b_i(C)\d_{ij}.
\end{equation}

Now choose $x=s$ and $z=t$ so that we have
\begin{equation*}
\sum_{g\in G}a_{ss}^{(i)}\left(h_1g\right)a_{tt}^{(j)}\left(h_2^{-1}g^{-1}\right)=a_{st}^{(i)}\left(h_1\right)a_{ts}^{(j)}\left(h_2^{-1}\right)b_i(C)\d_{ij}.
\end{equation*}
This becomes a statement about the group characters if we sum on both $s$ and $t$ to obtain
\begin{eqnarray*}
\sum_{\substack{1\leq s\leq m_i\\1\leq t\leq m_j}}\sum_{g\in G}a_{ss}^{(i)}\left(h_1g\right)a_{tt}^{(j)}\left(h_2^{-1}g^{-1}\right)&=&\sum_{g\in G}\left[\sum_{1\leq s\leq m_i}a_{ss}^{(i)}\left(h_1g\right)\right]\left[\sum_{1\leq t\leq m_j}a_{tt}^{(j)}\left(h_2^{-1}g^{-1}\right)\right]
\end{eqnarray*}
on the left hand side and
\begin{eqnarray*}
\sum_{\substack{1\leq s\leq m_i\\1\leq t\leq m_j}}a_{st}^{(i)}\left(h_1\right)a_{ts}^{(j)}\left(h_2^{-1}\right)b_i(C)\d_{ij}&=&b_i(C)\d_{ij}\sum_{1\leq s\leq m_i}\left[\sum_{1\leq t\leq m_j}a_{st}^{(i)}\left(h_1\right)a_{ts}^{(j)}\left(h_2^{-1}\right)\right]\\
&=&b_i(C)\d_{ij}\sum_{1\leq s\leq m_i}a_{ss}^{(i)}\left(h_1h_2^{-1}\right)
\end{eqnarray*}
on the right.  By definition, since $\chi\left(g^{-1}\right)=\overline{\chi(g)}$, we obtain
\begin{equation}\label{almost}
\sum_{g\in G}\chi_i\left(h_1g\right)\overline{\chi_j\left(gh_2\right)}=\chi_i\left(h_1h_2^{-1}\right)b_i(C)\d_{ij}.
\end{equation}
Finally, we determine $b_i(C)$.  If $i\neq j$, then $b_i(C)=0$ by Schur's Lemma.  If $i=j$, then we set $h_1=h_2=1$ in (\ref{almost}) and apply (\ref{2}) to obtain
 $$|G|=\sum_{g\in G}\chi_i(g)\overline{\chi_i(g)}=b_i(C)m_i.$$  
 Therefore $b_i(C)=|G|/\mathrm{dim}\chi_i$, and this completes the proof.
\end{proof}

\subsection{$r$-character theory and the proof of Theorem~\ref{general}}

Here we recall some basic facts about $r$-characters, and we then use the results of the previous subsection to
prove Theorem~\ref{general}.

If $r\geq 2$, then (\ref{defn}) offers a recursive formula for $r$-characters.  For $r=2$ and 3, if $\chi$ is an irreducible character, then we find that
$\chi^{(2)}\left(g_1, g_2\right)=\chi\left(g_1\right)\chi\left(g_2\right)-\chi\left(g_1g_2\right)$
and
\begin{displaymath}
\begin{split}
\chi^{(3)}\left(g_1, g_2, g_3\right)&=\chi\left(g_1\right)\chi\left(g_2\right)\chi\left(g_3\right)-\chi\left(g_1\right)\chi\left(g_2g_3\right)-\chi\left(g_3\right)\chi\left(g_1g_2\right)\\
&\hspace{4cm}-\chi\left(g_2\right)\chi\left(g_1g_3\right)+\chi\left(g_1g_2g_3\right)+\chi\left(g_2g_1g_3\right).
\end{split}
\end{displaymath}
For dimension $r\geq2$, these characters can be identically zero (see \cite[p. 244]{HJ2}).

\begin{lemma}\label{vanish}
Let $G$ be a finite group.  If $\chi$ is an irreducible character of $G$ and $r>\mathrm{dim}\chi$, then $\chi^{(r)}\left(\underline{g}\right)=0$ for all $\underline{g}\in G^{(r)}$.
\end{lemma}

Generalizing (\ref{1}), we obtain the following lemma.

\begin{lemma}\label{zero}
Let $G$ be a finite group.  If $\chi$ is a nontrivial irreducible character of $G$, then for any integer $r\geq1$, we have that $$\sum_{\underline{g}\in G^{(r)}}\chi^{(r)}\left(\underline{g}\right)=0.$$
\end{lemma}

\begin{proof}
When $r=1$, the result is simply (\ref{1}).  Now, assume for $r\geq1$ that $$\sum_{g_1,\dots,g_r\in G}\chi_i^{(r)}\left(g_1,\dots,g_r\right)=0.$$   Since $G^{(r+1)}=G\times G^{(r)}$, (\ref{defn}) implies that
\begin{align*}
&\sum_{\left(g_1,\dots,g_{r+1}\right)\in G^{(r+1)}}\chi_i^{(r+1)}\left(g_1,\dots,g_{r+1}\right)=\sum_{g_1\in G}\chi_i\left(g_1\right)\sum_{\left(g_2,\dots,g_{r+1}\right)\in G^{(r)}}\chi_i^{(r)}\left(g_2,\dots,g_{r+1}\right)\\
&\hspace{1.5cm}-\sum_{g_1\in G}\Bigg[\sum_{\left(g_2,\dots,g_{r+1}\right)\in G^{(r)}}\chi_i^{(r)}\left(g_1g_2,g_3,\dots,g_{r+1}\right)-\sum_{\left(g_2,\dots,g_{r+1}\right)\in G^{(r)}}\chi_i^{(r)}\left(g_2,g_1g_3,\dots,g_{r+1}\right)\\
&\hspace{8cm}-\cdots-\sum_{\left(g_2,\dots,g_{r+1}\right)\in G^{(r)}}\chi_i^{(r)}\left(g_2,g_3,\dots,g_1g_{r+1}\right)\Bigg].
\end{align*}
The bracketed expression inside the sum over $g_1$ is the sequential shift of the location of $g_1$ through the elements $g_2,\dots,g_{r+1}$.  The induction hypothesis and the observation that $g_1g_j$ varies over $G$ as $g_j$ varies over $G$ then imply the result.
\end{proof}

We are now able to prove Theorem~\ref{general}.

\begin{proof}[Proof of Theorem~\ref{general}]

If $i\neq j$, then it follows from \cite[Theorem 2.1]{J} that the sum is zero.  Also, if $G$ is abelian, then $G$ has only one-dimensional characters, so the sum is zero for all $r>1$ by Lemma \ref{vanish}.

For the remainder of the proof, we assume that $G$ is non-abelian and that $i=j$, and we let $\chi=\chi_i$ for simplicity.  We prove Theorem \ref{general} by writing the $r$-character $\chi^{(r)}$ in terms of the action of the symmetric group $S_r$ on products of $\chi$-values.  For $\sigma\in S_r$, let $n(\sigma)$ be the number of disjoint cycles in $\sigma$, including 1-cycles, and denote
\begin{equation}\label{newsigma}
\sigma=\left(a_1^\sigma(1),\dots,a_1^\sigma\left(k_1^\sigma\right)\right)\left(a_2^\sigma(1),\dots,a_2^\sigma\left(k_2^\sigma\right)\right)\cdots\left(a_{n(\sigma)}^\sigma(1),\dots,a_{n(\sigma)}^\sigma\left(k_{n(\sigma)}^\sigma\right)\right).
\end{equation}
The cycles have length $k_1^\sigma,k_2^\sigma,\dots,k_{n(\sigma)}^\sigma$, and as sets
\begin{equation*}
\{1,2,\dots,r\}=\left\{a_1^\sigma(1),\dots,a_1^\sigma\left(k_1^\sigma\right),a_2^\sigma(1),\dots,a_2^\sigma\left(k_2^\sigma\right),\dots,a_{n(\sigma)}^\sigma(1),\dots,a_{n(\sigma)}^\sigma\left(k_{n(\sigma)}^\sigma\right)\right\}.
\end{equation*}
With this notation, it is easy to see that (\ref{defn}) (also see p. 301 of \cite{J}) can be iterated to obtain the following formulas for values of $r$-characters as products of $\chi$-values.  We abuse notation and write $a$ for $g_a$.
\begin{align*}
\chi^{(r)}\left(g_1,\dots,g_r\right)=\sum_{\sigma\in S_r}\mathrm{sgn}(\sigma)\chi\left(a_1^\sigma(1)\cdots a_1^\sigma\left(k_1^\sigma\right)\right)\cdots\chi\left(a_{n(\sigma)}^\sigma(1)\cdots a_{n(\sigma)}^\sigma\left(k_{n(\sigma)}^\sigma\right)\right).
\end{align*}
Using the notation above, the sum in Theorem \ref{general} can now be rewritten, where $\underline{g}=\left(g_1,\dots,g_r\right)$, as
\begin{align}
\Omega&:=\sum_{\underline{g}\in G^{(r)}}\chi^{(r)}\left(\underline{g}\right)\overline{\chi^{(r)}\left(\underline{g}\right)}=\sum_{\underline{g}=\left(g_1,\dots,g_r\right)\in G^{(r)}}\chi^{(r)}\left(g_1,\dots,g_r\right)\overline{\chi^{(r)}\left(g_1,\dots,g_r\right)}\nonumber\\
\label{star}&=\sum_{\sigma,\tau\in S_r}\mathrm{sgn}(\sigma)\mathrm{sgn}(\tau)\sum_{\underline{g}\in G^{(r)}}\chi\left(a_1^\sigma(1)\cdots a_1^\sigma\left(k_1^\sigma\right)\right)\cdots\chi\left(a_{n(\sigma)}^\sigma(1)\cdots a_{n(\sigma)}^\sigma\left(k_{n(\sigma)}^\sigma\right)\right)\\
&\hspace{5cm}\times\overline{\chi\left(a_1^\tau(1)\cdots a_1^\tau\left(k_1^\tau\right)\right)}\cdots\overline{\chi\left(a_{n(\tau)}^\tau(1)\cdots a_{n(\tau)}^\tau\left(k_{n(\tau)}^\tau\right)\right)}.\nonumber
\end{align}
Now observe without loss of generality that we may order the cycles so that $g_r$ is in the last cycles $\left(a_{n(\sigma)}^\sigma(1),\dots,a_{n(\sigma)}^\sigma\left(k_{n(\sigma)}^\sigma\right)\right)$ and $\left(a_{n(\tau)}^\tau(1),\dots,a_{n(\tau)}^\tau\left(k_{n(\tau)}^\tau\right)\right)$.  Therefore, it follows that
\begin{align*}
\Omega=\sum_{\sigma,\tau\in S_r}\mathrm{sgn}(\sigma)\mathrm{sgn}(\tau)&\sum_{g_1,\dots,g_{r-1}\in G}\bigg[\chi\left(a_1^\sigma(1)\cdots a_1^\sigma\left(k_1^\sigma\right)\right)\cdots\chi\left(a_{n(\sigma)-1}^\sigma(1)\cdots a_{n(\sigma)-1}^\sigma\left(k_{n(\sigma)-1}^\sigma\right)\right)\\
&\hspace{1cm}\times\overline{\chi\left(a_1^\tau(1)\cdots a_1^\tau\left(k_1^\tau\right)\right)}\cdots\overline{\chi\left(a_{n(\tau)-1}^\tau(1)\cdots a_{n(\tau)-1}^\tau\left(k_{n(\tau)-1}^\tau\right)\right)}\bigg]\\
&\hspace{1cm}\times\sum_{g_r\in G}\chi\left(a_{n(\sigma)}^\sigma(1)\cdots a_{n(\sigma)}^\sigma\left(k_{n(\sigma)}^\sigma\right)\right)\overline{\chi\left(a_{n(\tau)}^\tau(1)\cdots a_{n(\tau)}^\tau\left(k_{n(\tau)}^\tau\right)\right)}.
\end{align*}
This last inner sum on $g_r$ can be evaluated by Lemma \ref{aux}.  We assume without loss of generality that the sum over $g_r$ is of the form $\sum_{g_r}\chi\left(A(\sigma)\cdot g_r\right)\overline{\chi\left(g_r\cdot A(\tau)\right)},$ where we use $A(\sigma),A(\tau)$ to denote the products of the remaining elements of $G$ in this particular sum which of course depend on $\sigma$ and $\tau$ (respectively).  Lemma \ref{aux} then eliminates $g_r$ from the sum and results in $\chi\left(A(\sigma)\cdot A(\tau)^{-1}\right)$ multiplied by $|G|/\mathrm{dim}\chi$.  This leaves a sum on $g_1,\dots,g_{r-1}$, where each of these elements appears in exactly one $\chi$ and exactly one $\overline{\chi}$, before possible cancellations.  If applying Lemma \ref{aux} results in the cancellation of a group element $\big($for example, if the rightmost element of $A(\sigma)$ is the inverse of the leftmost element of $A(\tau)^{-1}\big)$, then the sum over that group element is simply the sum of 1 over all elements in the group, so it contributes $|G|$.  

To complete the proof, we repeat this argument one-by-one, first with $g_{r-1}$, then $g_{r-2}$, and so on.
By applying the appropriate 1-character orthogonality relation for each of the remaining inner sums, we find that if we write the product $\sigma\tau^{-1}=x_1x_2x_3\cdots$ into disjoint cycles, and if we define $$m(\sigma,\tau):=\sum_{1\leq j\leq n\left(\sigma\tau^{-1}\right)}\left[\mathrm{length}\left(x_j\right)-1\right],$$
then we have that
\begin{equation}\label{inter}
\Omega=\sum_{\sigma,\tau\in S_r}\mathrm{sgn}(\sigma)\mathrm{sgn}(\tau)\frac{|G|^r}{(\mathrm{dim}\chi)^{m(\sigma,\tau)}}.
\end{equation}

It remains to show that 
\begin{equation}\label{formula}
\sum_{\sigma,\tau\in S_r}\mathrm{sgn}(\sigma)\mathrm{sgn}(\tau)\frac{|G|^r}{(\mathrm{dim}\chi)^{m(\sigma,\tau)}}=\frac{r!|G|^r}{(\mathrm{dim}\chi)^{r-1}}(\mathrm{dim}\chi-1)\cdots\big(\mathrm{dim}\chi-(r-1)\big).
\end{equation}
Recall that the Stirling numbers of the first kind \cite[(3.5.2), p.82]{Wilf}, denoted $\stirlingI{n}{k}$, count the number of permutations of $n$ with exactly $k$ disjoint cycles and are defined by the generating function $$\sum_{k=0}^n(-1)^{n-k}\stirlingI{n}{k}x^k=x(x-1)(x-2)\cdots(x-n+1).$$  The coefficient of $(\mathrm{dim}\chi)^i$ in the product $(\mathrm{dim}\chi-1)\cdots(\mathrm{dim}\chi-(r-1))$ is $(-1)^{r-i}\stirlingI{r}{i}/\mathrm{dim}\chi$, and so the coefficient of $1/(\mathrm{dim}\chi)^{r-i}$ on the right-hand side of (\ref{formula}) is $$(-1)^{r-i}\stirlingI{r}{i}r!|G|^r.$$

Now, it suffices to show that the coefficient of $1/(\mathrm{dim}\chi)^{r-i}$ in the sum
\begin{equation}\label{LHSsum}
\sum_{\sigma,\tau\in S_r}\mathrm{sgn}(\sigma)\mathrm{sgn}(\tau)\frac{1}{(\mathrm{dim}\chi)^{m(\sigma,\tau)}}
\end{equation}
is $(-1)^{r-i}\stirlingI{r}{i}r!$.  We rewrite $m(\sigma,\tau)$ as $$m(\sigma,\tau)=\sum_{1\leq j\leq n\left(\sigma\tau^{-1}\right)}\mathrm{length}\left(x_j\right)-n\left(\sigma\tau^{-1}\right)=r-n\left(\sigma\tau^{-1}\right),$$ since $\sigma\tau^{-1}=x_1x_2\cdots x_{n\left(\sigma\tau^{-1}\right)}$ is a product of disjoint cycles, including fixed points.
We must evaluate the term of the sum \eqref{LHSsum} corresponding to $n\left(\sigma\tau^{-1}\right)=i$.  For fixed $\sigma\in S_r$, the number of $\tau\in S_r$ such that $n\left(\sigma\tau^{-1}\right)=i$ is equal to the number of $\tau\in S_r$ that can be written as a product of $i$ disjoint cycles, since $\left\{\sigma\tau^{-1}:\tau\in S_r\right\}=S_r$ as sets.  Therefore, for fixed $\sigma$, there are $\stirlingI{r}{i}$ such $\tau\in S_r$.  Since there are $r!$ possibilities for $\sigma$, we have that there are a total of $\stirlingI{r}{i}r!$ pairs $(\sigma,\tau)\in S_r^{(2)}$ such that $n\left(\sigma\tau^{-1}\right)=i$.
We now show that the product of signatures $\mathrm{sgn}(\sigma)\mathrm{sgn}(\tau)$ for each such pair $(\sigma,\tau)$ equals $(-1)^{r-i}$.  Since $\mathrm{sgn}(\sigma)\mathrm{sgn}(\tau)=\mathrm{sgn}\left(\sigma\tau^{-1}\right)$, it suffices to evaluate $\mathrm{sgn}\left(\sigma\tau^{-1}\right)$ for each such $(\sigma,\tau)$.
Suppose that in the decomposition $\sigma\tau^{-1}=x_1x_2\cdots x_i$, there are $n_\ell$ $\ell$-cycles for each $1\leq\ell\leq r$ and one additional cycle, say $x_k$, that is not written in this way.  Then we have that $\sum_{1\leq\ell\leq r}n_\ell=i-1$, and the composition of all of the $\ell$-cycles with $x_k$ makes up $\sigma\tau^{-1}$.
The composition of all $\ell$-cycles except $x_k$ has signature $(-1)^{\sum_{2\leq\ell\leq r}n_\ell(\ell-1)},$ where the sum begins with $\ell=2$ since fixed points do not contribute to signature.  The remaining cycle $x_k$ has length $r-\left(\sum_{1\leq\ell\leq r}n_\ell\ell\right)$, so its signature is $(-1)^{r-\left(\sum_{1\leq\ell\leq r}n_\ell\ell\right)-1}$.
Therefore, we have that $$\mathrm{sgn}\left(\sigma\tau^{-1}\right)=(-1)^{\sum_{2\leq\ell\leq r}n_\ell(\ell-1)+r-\left(\sum_{1\leq\ell\leq r}n_\ell\ell\right)-1}=(-1)^{r-i}.$$
Thus, the coefficient of $1/(\mathrm{dim}\chi)^{r-i}$ in the sum \eqref{LHSsum} is $$(-1)^{r-i}\stirlingI{r}{i}r!.$$  This completes the proof.
%If $1\leq i\leq r$, then  the coefficient of $(\mathrm{dim}\chi)^{r-i}$ on the right hand side is $$(-1)^{i-1}\cdot\frac{r!|G|^r}{(\mathrm{dim}\chi)^{r-1}}\cdot\frac{r(r-1)\cdots(r-(i-1))}{i}.$$  We will now show that the left hand sum gives the same coefficient.  Clearly, if we fix an element $g\in S_r$, then $\{gh:h\in S_r\}=S_r$ as sets.  Then the number of $i$-cycles in $\{\sigma\tau^{-1}:\sigma\in S_r,\tau\in S_r\}$ equals the product of the number of $i$-cycles in $S_r$ with the number of elements in $S_r$.  Since $\mathrm{sgn}(\sigma)\mathrm{sgn}(\tau)$ contributes $(-1)^{i-1}$ for each $i$-cycle appearing in $\sigma\tau^{-1}$, and since $$\frac{1}{(\mathrm{dim}\chi)^{i-1}}=\frac{1}{(\mathrm{dim}\chi)^{r-1}}\cdot(\mathrm{dim}\chi)^{r-i},$$ we see that the coefficient of $(\mathrm{\dim}\chi)^{r-i}$ on the left hand side is $$(-1)^{i-1}\cdot\frac{r!|G|^r}{(\mathrm{dim}\chi)^{r-1}}\cdot\frac{r(r-1)\cdots(r-(i-1))}{i}.$$  This completes the proof.
\end{proof}

\section{Proofs of Theorems \ref{moonshine} and \ref{multiplicities}}\label{moonshineproofs}

We now prove Theorems \ref{moonshine} and \ref{multiplicities}.  Theorem \ref{moonshine} guarantees that weak moonshine can be extended to width $s$.  Theorem \ref{multiplicities} shows that the higher width orthogonality relations for Frobenius $r$-characters are compatible with width $s$ weak moonshine.  Namely, we show how to determine the multiplicity generating functions for the representation space for each nontrivial $\rho_i$ in the graded $G$-module $V_G$ using the higher width McKay-Thompson series.

\subsection{Proof of Theorem \ref{moonshine}}

By the Schur orthogonality relations for 1-characters, the multiplicity generating functions are given by
\begin{equation}\label{modular}
\mathcal{M}_i(\tau):=\sum_{n\gg-\infty}m_i(n)q^n=\sum_{n\gg-\infty}\frac{1}{|G|}\sum_{g\in G}\overline{\chi_i(g)}\mathrm{Frob}_1(g;n)q^n=\frac{1}{|G|}\sum_{g\in G}\overline{\chi_i(g)}T(1,g;\tau).
\end{equation}

Theorem 1.1 of \cite{REU} guarantees that  there is a $G$-module $V_G=\bigoplus_n V_G(n)$, $q$-graded traces $T(1,g;\tau)$ which are modular functions for all $g\in G$, and non-negative integer multiplicities $m_i(n)$ for the representation spaces of each $\rho_i$ in $V_G(n)$.  Moreover, the results in Section 5 of \cite{REU}  guarantee that $V_G$ can be chosen to be asymptotically regular.
Therefore,
it suffices to construct the McKay-Thompson series for $V_G^{(r)}$ for $r>1$.  By the definitions of the generalized graded trace functions and the $r$-Frobenius of $\underline{g}\in G^{(r)}$ on $V_G^{(r)}(n)$, for each $1\leq r\leq s$ we have 
\begin{align*}
T\left(r,\underline{g};\tau\right)&=\sum_{n\gg-\infty}\mathrm{Frob}_r\left(\underline{g};n\right)q^n=\sum_{n\gg-\infty}\sum_{1\leq j\leq t}m_j(n)\chi_j^{(r)}\left(\underline{g}\right)q^n=\sum_{1\leq j\leq t}\chi_j^{(r)}\left(\underline{g}\right)\mathcal{M}_j(\tau).
\end{align*}
Since all of the $\mathcal{M}_j(\tau)$ are modular functions by (\ref{modular}), we must have that the $T\left(r,\underline{g};\tau\right)$ are modular functions as well for each $\underline{g}\in G^{(r)}$.
\qed

\subsection{Proof of Theorem \ref{multiplicities}}
The multiplicity generating functions $\mathcal{M}_i(\tau)=\sum_{n\gg -\infty} m_i(n)q^n$ may be expressed in terms of the $T\left(r,\underline{g};\tau\right)$ when  $\dim \chi_i \geq r$.  By Theorem \ref{general}, we have that
\begin{align*}
\mathcal{M}_i(\tau)&
=\sum_{n\gg-\infty}\frac{\left(\mathrm{dim}\chi_i\right)^{r-1}}{r!|G|^r\left(\mathrm{dim}\chi_i-1\right)\cdots\left(\mathrm{dim}\chi_i-(r-1)\right)}\sum_{\underline{g}\in G^{(r)}}\overline{\chi_i^{(r)}\left(\underline{g}\right)}\sum_{1\leq j\leq t}m_j(n)\chi_j^{(r)}\left(\underline{g}\right)q^n\\
&=\sum_{n\gg-\infty}\frac{\left(\mathrm{dim}\chi_i\right)^{r-1}}{r!|G|^r\left(\mathrm{dim}\chi_i-1\right)\cdots\left(\mathrm{dim}\chi_i-(r-1)\right)}\sum_{\underline{g}\in G^{(r)}}\overline{\chi_i^{(r)}\left(\underline{g}\right)}\mathrm{Frob}_r\left(\underline{g};n\right)q^n\\
&=\frac{\left(\mathrm{dim}\chi_i\right)^{r-1}}{r!|G|^r\left(\mathrm{dim}\chi_i-1\right)\cdots\left(\mathrm{dim}\chi_i-(r-1)\right)}\sum_{\underline{g}\in G^{(r)}}\overline{\chi_i^{(r)}\left(\underline{g}\right)}T\left(r,\underline{g};\tau\right).
\end{align*}
Therefore, the number of copies of the representation space for $\rho_i$ in all of the graded components $V_G^{(r)}(n)$ for all $1\leq r\leq\mathrm{dim}\chi_i$ are given as the Fourier coefficients of the above linear combination of the modular McKay-Thompson series.
\qed

\section{Example: $D_4$ and $Q_8$}\label{example}

The dihedral group $D_4$ and the quaternion group $Q_8$, given by
$$D_4=\{1,r,r^2,r^3,s,rs,r^2s,r^3s\}\ \ \  {\text {\rm and}}\ \ \ Q_8=\{1,-1,i,-i,j,-j,k,-k\},$$
have the same character table.
%\vspace{.1cm}

\begin{center}
\begin{table}[H]
\caption{Character Table for $D_4$ and $Q_8$}
\begin{tabular}{|c|ccccc|}
\hline
$\mathbf{D_4}$&$\{1\}$&$\{r^2\}$&$\{r,r^3\}$&$\{s,r^2s\}$&$\{rs,r^3s\}$\\
\hline
$\mathbf{Q_8}$&$\{1\}$&$\{-1\}$&$\{i,-i\}$&$\{j,-j\}$&$\{k,-k\}$\\
\hline
\hline
&$C_1$&$C_2$&$C_3$&$C_4$&$C_5$\\
\hline
$\chi_1$&1&1&1&1&1\\
$\chi_2$&1&1&$-1$&1&$-1$\\
$\chi_3$&1&1&$-1$&$-1$&1\\
$\chi_4$&1&1&1&$-1$&$-1$\\
$\chi_5$&2&$-2$&0&0&0\\
\hline
\end{tabular}
\end{table}
%\caption{Character Table for $D_4$ and $Q_8$}
\end{center}
%\vspace{.1cm}
We abused notation by letting $C_1, C_2, C_3, C_4,$ and $C_5$ denote the five conjugacy classes of both $D_4$ and $Q_8$.
In both cases $C_1$ and $C_2$ contain a single group element, while the other conjugacy classes contain 2 group elements.

These groups share weak moonshine where the McKay-Thompson series are the Hauptmoduln for $\Gamma_0(1),$ $\Gamma_0(2),$ and $\Gamma_0(4)$.  For convenience, we let $f_1(\tau)$, $f_2(\tau)$, and $f_4(\tau)$ be these Hauptmoduln, which are given by
\begin{align*}
f_1(\tau)&:=J(\tau)=q^{-1}+196884q+21493760q^2+864299970q^3+20245856256q^4+O\left(q^5\right),\\
f_2(\tau)&:=\left(\frac{\eta(\tau)}{\eta(2\tau)}\right)^{24}+24=q^{-1}+276q-2048q^2+11202q^3-49152q^4+184024q^5+O\left(q^6\right),\\
f_4(\tau)&:=\left(\frac{\eta(\tau)}{\eta(4\tau)}\right)^8+8=q^{-1}+20q-62q^3+216q^5-641q^7+1636q^9-3778q^{11}+O\left(q^{13}\right).
\end{align*}
Here $\eta(\tau):=q^{1/24}\prod_{n=1}^{\infty}(1-q^n)$ is Dedekind's eta-function.  For each of the five conjugacy classes $C_j$, denote the McKay-Thompson series corresponding to any $g\in C_j$ by $T\left(1,C_j;\tau\right)$.  For $D_4$, choose $T\left(1,C_j;\tau\right)$  to be the Hauptmodul $f_j(\tau)$ of level $\mathrm{ord}(g)$ for elements $g\in C_j$.  
For both groups, we have the following McKay-Thompson series:
\begin{displaymath}
\begin{split}
T\left(1,C_1;\tau\right)&:=f_1(\tau),\ \ 
T\left(1,C_2;\tau\right):=f_2(\tau),\ \ 
T\left(1,C_3;\tau\right):=f_4(\tau),\\
&T\left(1,C_4;\tau\right):=f_2(\tau),\ \
T\left(1,C_5;\tau\right):=f_2(\tau).
\end{split}
\end{displaymath}
For each $1\leq i\leq 5$, we use (\ref{2}) to compute the generating function $\mathcal{M}_i(\tau)$ of the multiplicities $m_i(n)$ of the representation spaces for each $\rho_i$ in $V_{D_4}(n)$ and $V_{Q_8}(n)$ to be
\begin{align*}
\mathcal{M}_i(\tau)=\frac{1}{8}\left(f_1\chi_i\left(C_1\right)+f_2\overline{\chi_i\left(C_2\right)}+2f_4\overline{\chi_i\left(C_3\right)}+2f_2\overline{\chi_i\left(C_4\right)}+2f_2\overline{\chi_i\left(C_5\right)}\right),
\end{align*}
where $\chi_i\left(C_j\right)$ denotes the value of the character $\chi_i$ at any element $g\in C_j$.  The first few terms of each multiplicity generating function for both $D_4$ and $Q_8$ are given below:
\begin{align*}
\mathcal{M}_1(\tau)&=q^{-1}+24788q+2685440q^2+108044482q^3+O\left(q^4\right),\\
\mathcal{M}_2(\tau)&=24640q+2686464q^2+108038912q^3+O\left(q^4\right),\\
\mathcal{M}_3(\tau)&=24640q+2686464q^2+108038912q^3+O\left(q^4\right),\\
\mathcal{M}_4(\tau)&=24512q+2687488q^2+108033280q^3+O\left(q^4\right),\\
\mathcal{M}_5(\tau)&=49152q+5373952q^2+216072192q^3+O\left(q^4\right).
\end{align*}
Note that $\mathcal{M}_2(\tau)=\mathcal{M}_3(\tau)$.  It is not difficult to prove
(for example, see \cite{DGO1})  that this moonshine is asymptotically regular.
More precisely, if $1\leq i\leq 5$, then let
$$
\delta_i(n):=\frac{m_i(n)}{m_1(n)+m_2(n)+m_3(n)+m_4(n)+m_5(n)}.
$$
The asymptotic regularity is given by
$$
\lim_{n\rightarrow +\infty}\delta_i(n)=\frac{\dim \chi_i}{\sum_{j=1}^5 \dim \chi_j}
=\begin{cases} \frac{1}{6} \ \ \ \ \ &{\text {\rm if}}\ 1\leq i\leq 4,\\
\frac{1}{3} \ \ \ \ \ &{\text {\rm if}}\ i=5.
\end{cases}
$$
The table below illustrates the rapid convergence exhibited by this moonshine.

\vspace{.1cm}
\begin{table}[H]\label{tab:intro-deltas}
\caption{Asymptotic Distributions}
\begin{tabular}{|c|c|c|c|c|c|}\hline
$n$&$\delta_1(n)$&$\delta_2(n)=\delta_3(n)$&$\delta_4(n)$& $\delta_5(n)$
\\\hline 
$1$& $0.16779\dots$     &  $0.16678\dots$ & $0.16592\dots$ &  $0.33271\dots$ \\ 
$2$& $0.16659\dots$     &  $0.16665\dots$ & $0.16671\dots$ &  $0.33337\dots$ \\ 
$3$& $0.16666\dots$     &  $0.16666\dots$ & $0.16665\dots$ &  $0.33332\dots$ \\ 
$4$& $0.16666\dots$     &   $0.16666\dots$ & $0.16666\dots$ &  $0.33333\dots$ \\ 
$\vdots$& $\vdots$     &  $\vdots$ & $\vdots$ &  $\vdots$ \\ 
\hline
\end{tabular}
\end{table}
\vspace{.1cm}

To illustrate Theorem \ref{moonshine}, we now extend to width 2 weak moonshine.  We consider the 2-character tables and McKay-Thompson series at corresponding pairs of elements $\left(g_1,g_2\right)$ in $D_4^{(2)}$ and $Q_8^{(2)}$.  The 2-character tables contain four rows of zeros corresponding to $\chi_1,\dots,\chi_4$ which have dimension 1, and one possibly non-zero row corresponding to the values of $$\chi_5^{(2)}\left(g_1,g_2\right):=\chi_5\left(g_1\right)\chi_5\left(g_2\right)-\chi_5\left(g_1g_2\right).$$
All of the values in the 2-character tables of $D_4$ and $Q_8$ are identical except eight, which are lined up below based on the elements' conjugacy classes:
\begin{eqnarray*}
\chi_5^{(2)}(s,r^2s)=2,&\hspace{.5cm}&\chi_5^{(2)}(j,-j)=-2,\\
\chi_5^{(2)}(s,s)=-2,&\hspace{.5cm}&\chi_5^{(2)}(j,j)=2,\\
\chi_5^{(2)}(r^2s,s)=2,&\hspace{.5cm}&\chi_5^{(2)}(-j,j)=-2,\\
\chi_5^{(2)}(r^2s,r^2s)=-2,&\hspace{.5cm}&\chi_5^{(2)}(-j,-j)=2,\\
\chi_5^{(2)}(rs,r^3s)=2,&\hspace{.5cm}&\chi_5^{(2)}(k,-k)=-2,\\
\chi_5^{(2)}(rs,rs)=-2,&\hspace{.5cm}&\chi_5^{(2)}(k,-k)=2,\\
\chi_5^{(2)}(r^3s,rs)=2,&\hspace{.5cm}&\chi_5^{(2)}(k,-k)=-2,\\
\chi_5^{(2)}(r^3s,r^3s)=-2,&\hspace{.5cm}&\chi_5^{(2)}(k,-k)=2.
\end{eqnarray*}
We now illustrate that the width 2 McKay-Thompson series are different for these two groups.  Consider the pair $\left(r^3s,rs\right)\in D_4^{(2)}$.  Its McKay-Thompson series for $V_{D_4}^{(2)}$ is given by
\begin{align*}
T\left(2,\left(r^3s,rs\right);\tau\right)&=\sum_{1\leq i\leq 5}\chi_i^{(2)}\left(r^3s,rs\right)\mathcal{M}_i(\tau)=\chi_5^{(2)}\left(r^3s,rs\right)\mathcal{M}_5(\tau)\\
&=98304q+10747904q^2+432144384q^3+O\left(q^4\right).
\end{align*}
The McKay-Thompson series of the corresponding pair $(-k,k)\in Q_8^{(2)}$ for $V_{Q_8}^{(2)}$ is given by
\begin{align*}
T\left(2,(-k,k);\tau\right)&=\sum_{1\leq i\leq 5}\chi_i^{(2)}(-k,k)\mathcal{M}_i(\tau)=\chi_5^{(2)}(-k,k)\mathcal{M}_5(\tau)\\
&=-98304q-10747904q^2-432144384q^3+O\left(q^4\right).
\end{align*}
Therefore, width 2 weak moonshine distinguishes $D_4$ and $Q_8$.

\end{document}